\documentclass[11pt,oneside]{amsart}
\usepackage[foot]{amsaddr}
\usepackage{bm, calc,geometry, verbatim, graphicx, amssymb, color, mathabx, mathtools, enumitem, bm, mathrsfs, amsmath, amsthm, ulem, xspace,tabularx }
\usepackage[usenames,dvipsnames,svgnames,table]{xcolor}
\usepackage[pdftex,bookmarks,colorlinks,breaklinks]{hyperref} 

\hypersetup{linkcolor=blue,citecolor=red,filecolor=dullmagenta,urlcolor=darkblue} 
\newcommand\brho{\operatorname{\boldsymbol{\rho}}}

\newtheorem{theorem}{Theorem}[section]
\newtheorem{corollary}[theorem]{Corollary}
\newtheorem{lemma}[theorem]{Lemma}

\newtheorem{remark}[theorem]{Remark}

\newtheorem{question}{Problem} 
\newcommand{\Mod}[1]{\ (\mathrm{mod}\ #1)}

\newcommand\mult{\operatorname{\textup{{\fontfamily{ptm}\selectfont mult}}}}
\newcommand\dg{\operatorname{\textup{{\fontfamily{ptm}\selectfont deg}}}}

\usepackage{physics}

\newcommand\roundup[1]{\left\lceil#1\right\rceil}
\newcommand\rounddown[1]{\left\lfloor#1\right\rfloor}

      \makeatletter 
      \def\@setcopyright{}
      \def\serieslogo@{}
      \makeatother      

\begin{document}
   \author{Amin Bahmanian}   \address{Department of Mathematics,
  Illinois State University, Normal, IL USA 61790-4520}
  \email{mbahman@ilstu.edu}
  
   \author{Anna Johnsen-Yu}
  \address{Department of Mathematics and Statistics,  Georgia State University, Atlanta, GA USA 30303-2918}
  \address{Department of Mathematics, Vanderbilt University, Nashville, TN USA 37240-0001}
  \email{anna.c.yu@vanderbilt.edu}

\title[The Andersen-Hoffman Theorem for Equitable Rectangles] {The Andersen-Hoffman Theorem for Equitable Rectangles}

\dedicatory{In memory of Dean Gunnar Hoffman (1949 -- 2022)}

\begin{abstract}  
More than forty years ago, Andersen and Hoffman independently proved that every symmetric Latin rectangle can be extended to a symmetric Latin square with prescribed diagonal entries. We generalize this theorem as follows. Let $k\le n^2$, and let $M$ be an $n\times n$ array whose top-left $r\times r$ subarray is filled with symbols from $\{1,2,\ldots,k\}$. Suppose that, for each $i\in\{1,\ldots,r\}$ and each symbol, the number of occurrences of that symbol in row $i$ equals its number of occurrences in column $i$, and that each remaining diagonal entry is either empty or already contains a symbol from $\{1,\ldots,k\}$. We establish necessary and sufficient conditions for completing $M$ so that the resulting array is symmetric off the prescribed $r\times r$ subarray, each symbol occurs a specified total number of times in $M$, and, for every symbol, its numbers of occurrences in any two rows (respectively, columns) differ by at most one.

Restricted to symmetric arrays, our theorem generalizes results of Cruse (1974), Goldwasser et al.\ (2015), and Bahmanian and Hilton (2025). It also extends Baranyai's theorem for complete graphs (1973) by characterizing when a partial coloring of $K_r$ with a loop on every vertex can be extended to an almost regular coloring of $K_n$ with a loop on every vertex, where $n\ge r$.
\end{abstract}

   \subjclass[2010]{05B15}
   \keywords{Latin Square, Embedding, Cruse's Theorem, Andersen-Hoffman's Theorem, Ryser's Theorem, Equitable Rectangles}

   \maketitle   
   
\section{Introduction} 
An {\it $n\times n$ partial Latin square} is simply an $n\times n$ array such that each entry is either empty or contains a symbol $\ell\in\{1,\dots,n\}$ such that no symbol appears more than once in any row or column of the array. The broad question we are considering is this: under what conditions is it possible to extend an $n\times n$ partial Latin square to an $n\times n$ Latin square? 
Completing partial Latin squares has been shown to be NP-complete \cite{MR0739595}. In fact, even when we are restricted to symmetric arrays, the question is NP-complete \cite{MR0704260}. Nonetheless, progress has been made on this problem for partial Latin squares in which a rectangular subarray is completely filled and the remaining cells are empty.

In order to discuss this progress and state our main results, we include several definitions. In this paper, we always assume the following.
\begin{align*}
    n,k\in\mathbb{N}, && r\in\mathbb{N}\cup\{0\}, && r\leq n, && 1\leq k\leq n^2, && [n]:=\{1,\dots,n\}.
\end{align*}
Let $M$ be an $r\times r$ array whose cells $(i,j)$ each contain at most one symbol from the set $[k]$ of symbols for  $i\in[r],j\in[r]$. We refer to the symbol in cell $(i,j)$ of $M$ as $M(i,j)$ for $i\in[r], j\in[r]$, and define the following notation.
\begin{itemize}
    \item The number of occurrences of the symbol $\ell$ in $M$ is denoted by $|M_\ell|$ for $\ell\in[k]$.
    \item The number of occurrences of the symbol $\ell$ in a fixed row $i$ of $M$ and in a fixed column $j$ of $M$ are denoted by $|M_\ell^i|$ for $i\in[r], \ell\in[k]$ and by $|^jM_\ell|$ for $j\in[r],\ell\in[k]$, respectively.
\end{itemize}
Let $\brho = (\rho_1,\dots,\rho_k)$ be a $k$-tuple such that the following hold.
\begin{align*}
    1\leq \rho_\ell\leq n^2 \text{ for }\ell\in[k], && \sum_{\ell\in [k]} \rho_\ell=n^2.
\end{align*}
An $r\times r$ array $M$ is an 
{\it equitable $\brho$-Latin rectangle} if
\begin{itemize}
    \item Every cell in $M$ contains exactly one symbol $\ell\in[k]$,
    \item $|M_\ell|\leq \rho_\ell$ for $\ell\in[k]$, 
    \item $|M_\ell^i|\leq \roundup{\dfrac{\rho_\ell}{n}}$ for $\ell\in[k],i\in[r]$, and
    \item $|^jM_\ell|\leq \roundup{\dfrac{\rho_\ell}{n}}$ for $\ell\in[k],j\in[r]$.
\end{itemize}
An $n\times n$ array $M$ is an {\it equitable $\brho$-Latin square} if
\begin{itemize}
    \item Every cell in $M$ contains exactly one symbol $\ell\in[k]$,
    \item $|M_\ell| = \rho_\ell$ for $\ell\in[k]$, 
    \item $\rounddown{\dfrac{\rho_\ell}{n}}\leq |M_\ell^i| \leq \roundup{\dfrac{\rho_\ell}{n}}$ for $\ell\in[k],i\in[n]$, and
    \item $\rounddown{\dfrac{\rho_\ell}{n}}\leq |^jM_\ell| \leq \roundup{\dfrac{\rho_\ell}{n}}$ for $\ell\in[k],j\in[n]$.
\end{itemize}
If $\rho_\ell\leq n$ for $\ell\in[k]$, we call an equitable $\brho$-Latin rectangle (square) a {\it $\brho$-Latin rectangle (square)}, and when $\rho_\ell=n$ for $\ell\in[k]$, we call an equitable $\brho$-Latin rectangle (square) a {\it Latin rectangle (square)}. We note that a Latin rectangle can be viewed as a partial Latin square in which the top left subarray is entirely filled and the remainder of the square is empty.

An $r\times r$ array $M$ is {\it externally symmetric} if $|M_\ell^i| = |^iM_\ell|$ for $i\in[r]$, and is {\it symmetric} if it is symmetric with respect to the main diagonal. Note that symmetric arrays are externally symmetric.
Let $$N = \begin{bmatrix}
    M & A\\
    B & C
\end{bmatrix},$$
where $M$ is an $r\times r$ array, $A$ is an $r\times (n-r)$ array, $B$ is an $(n-r)\times r$ array, and $C$ is an $(n-r)\times(n-r)$ array.
If the cells in $N$ are filled so that $B=A^T$ and $C$ is symmetric with respect to its main diagonal, we say that $N$ is {\it symmetric off $M$}, where the array $M$ may or may not be symmetric. Note that if an $r\times r$ externally symmetric array $M$ is extended to an $n\times n$ array $N$ which is symmetric off $M$, then $N$ is also externally symmetric. 
Now let $d_1,\dots,d_k$ be integers such that \begin{align*}
    d_\ell\geq 0 \text{ for }\ell\in[k],&& \sum_{\ell\in[k]}d_\ell\leq n-r.
\end{align*}
We define the $k$-tuple $\vb*{d}:=(d_1,\dots, d_k)$ to be a {\it diagonal tail}.
If the cells in $N$ are filled so that there are at least $d_\ell$ occurrences of $\ell$ on the main diagonal of $C$ for $\ell\in[k]$, then we say that $N$ has the {\it prescribed diagonal tail $\vb*{d}$}.  We refer to the diagonal tail as the {\it diagonal} when $r=0$ and we say that an $n\times n$ array has the prescribed diagonal $\vb*{d}$ if there are at least $d_\ell$ occurrences of $\ell$ on its main diagonal. 
We note that by permuting the rows and corresponding columns of $N$, we may prescribe the order in which symbols appear on the diagonal in $C$ without changing the number of occurrences of any symbol in each row and column and without changing the number of occurrences of any symbol on the main diagonal. Hence, in order to extend an $r\times r$ array $M$ to an externally symmetric equitable $\brho$-Latin square $N$ that is symmetric off $M$ in which the order of the symbols in the diagonal tail $\vb*{d}$ are prescribed on the main diagonal of $N$ outside of $M$, it suffices to extend $M$ to an externally symmetric equitable $\brho$-Latin square that is symmetric off $M$ and has the diagonal tail $\vb*{d}$.

The study of Latin square completion dates back to Hall's theorem for extending Latin rectangles published in 1945 \cite{MR13111} and Ryser's well-known theorem for extending Latin rectangles to Latin squares published in 1951 \cite{MR42361}. 
In the 1970s, Cruse established conditions under which an $r\times r$ symmetric Latin rectangle can be extended to an $n\times n$ symmetric Latin square \cite{MR329925}.
In the 1980s, Andersen, Hilton, and Rodger published results on extending Latin rectangles (not necessarily symmetric) with prescribed diagonal tails to Latin squares \cite{MR667240}. Andersen and Hilton \cite{ANDERSEN1980125,ANDERSEN1980235} consider the construction, decomposition, and embedding of generalized Latin rectangles. 
Andersen and Hoffman also independently established conditions under which symmetric Latin rectangles with prescribed diagonal tails may be extended to Latin squares. Hoffman's theorem is as follows. 
\begin{theorem}\cite{MR694466} \label{thm:hoffman}
    Let $\vb*{d}=(d_1,\dots,d_n)$
    such that $\sum_{\ell\in[k]}d_\ell = n-r$. Then an $r\times r$ symmetric Latin rectangle $M$ on $[n]$ can be extended to an $n\times n$ symmetric Latin square on $[n]$ with the prescribed diagonal tail $\vb*{d}$ if and only if the following hold.
    \begin{align*}
        &|M_\ell|\geq 2r-n + d_\ell &&\text{ for } \ell\in[n],\\
        &|M_\ell| + d_\ell \equiv n\Mod 2 &&\text{ for } \ell\in[n].
    \end{align*}
\end{theorem}
Andersen's result is slightly more general, not requiring the prescribed diagonal tail to contain exactly $n-r$ symbols, and being stated for externally symmetric squares.
\begin{theorem}\cite[Theorem 11]{MR772580} \label{thm:andersen}
    Let $\vb*{d}=(d_1,\dots,d_n)$ such that $\sum_{\ell\in[n]}d_\ell\leq n-r$. Then an $r\times r$ externally symmetric Latin rectangle $M$ on $[n]$ can be extended to an $n\times n$ Latin square on $[n]$ which is symmetric off $M$ and has the prescribed diagonal tail $\vb*{d}$ if and only if the following hold.
        \begin{align*}
            &|M_\ell| \geq 2r-n+d_\ell \quad &&\text{ for } \ell \in [n],\\
            &|M_\ell|+d_\ell \equiv n \Mod 2 \quad &&\text{ for at least } r + \sum_{\ell\in[n]}d_\ell \text{ symbols } \ell\in[n].
        \end{align*}
\end{theorem}
Non-symmetric generalizations include a generalization of Hall's theorem for $\brho$-Latin squares by Goldwasser et al. \cite{MR3280683}, a generalization of Ryser's theorem for $\brho$-Latin squares \cite{MR4414830}, and a generalization of both of these results for equitable $\brho$-Latin squares \cite{RyserForEquitRhoLatin}.
Bahmanian and Hilton generalized Andersen-Hoffman's theorem for symmetric $\brho$-Latin squares \cite{MR5024856}. 
We generalize Theorem \ref{thm:andersen} and Theorem \ref{thm:hoffman} by establishing necessary and sufficient conditions that ensure an $r\times r$ externally symmetric equitable $\brho$-Latin rectangle $M$ can be extended to an $n\times n$ equitable $\brho$-Latin square which is symmetric off $M$ and has the prescribed diagonal tail $\vb*{d}$ (see Theorem \ref{mainthm}). Our main result has a number of consequences.  In light of the famous Baranyai’s theorem, which constructs almost regular colorings of complete uniform hypergraphs \cite[Theorem 1]{MR0416986}, our result when restricted to symmetric arrays may be viewed as extending partial colorings to almost regular colorings of $\mathbb{K}_n$, the complete graph on $n$ vertices in which each vertex is additionally incident to a 1-loop (a 1-loop contributes one to the degree of the vertex to which it is incident). 
If each symbol occurs at most $n$ times in rows and columns of the extended array, our result implies Bahmanian and Hilton's result for $\brho$-Latin rectangles \cite[Theorem 1.3]{MR5024856}. If the number of occurrences of each symbol in the array is divisible by $n$, our result additionally complements results for extending factorizations of complete graphs (see for example \cite{MR1315436,MR1983358,MR2325799,BAHMANIAN2025374}).

In order to state our main result, we introduce some more notation. Let $M$ be an $r\times r$ equitable $\brho$-Latin rectangle on $[k]$. 
We will always assume that $[r]$ is the set of rows in $M$.
Let $i\in[r]$ and $\ell\in[k]$ be a row in $M$ and a symbol, respectively, and let $I\subseteq [r]$ and $K\subseteq[k]$ be a subset of the rows in $M$ and a subset of symbols in $[k]$, respectively. Let $\overline{I}:=[r]\backslash I$ and $\overline{K}:=[k]\backslash K$ for $I\subseteq[r], K\subseteq[k]$. 
Recall that by definition, each symbol $\ell\in[k]$ may either occur $\roundup{\rho_\ell/n}$ or $\rounddown{\rho_\ell/n}$ times in any given row of an equitable $\brho$-Latin square $N$. If $n\mid \rho_\ell$, then there must be exactly $\rho_\ell/n$ occurrences of $\ell$ in every row of $N$, but if $n\nmid \rho_\ell$, each row of $N$ may either contain $\roundup{\rho_\ell/n}$ or $\rounddown{\rho_\ell/n}$ occurrences of $\ell$. For this reason, we refer to a symbol $\ell\in[k]$ as {\it forced} if $n\mid \rho_\ell$ and as {\it free} if $n\nmid \rho_\ell$. We define $\eta_K(i)$ to be the number of free symbols in $K$ which occur fewer than $\roundup{\rho_\ell/n}$ times in row $i$. Note that if the symbols in $K$ which occur fewer than $\roundup{\rho_\ell/n}$ times in row $i$ are forced symbols, we will have $\eta_K(i)=0$. For a free symbol $\ell\in[k]$, we define $\eta_I(\ell)$ to be the number of rows in $I$ where the symbol $\ell$ appears fewer than $\roundup{\rho_\ell/n}$ times. For a forced symbol $\ell\in[k]$, we define $\eta_I(\ell)$ to be $0$.

Suppose that an $r\times r$ externally symmetric equitable $\brho$-Latin rectangle $M$ is extended to an $n\times n$ equitable $\brho$-Latin square $N$ which is symmetric off $M$ and has the prescribed diagonal tail $\vb*{d}$. We may assume that the symbols contained in $M$ are found in the top left $r\times r$ subarray of $N$ and that there are at least $d_\ell$ occurrences of $\ell$ in the diagonal entries in rows $i\in\{r+1,\dots,n\}$ of $N$ for $\ell\in[k]$. Let us fix a symbol $\ell\in[k]$. 
Clearly, we have that
\begin{align}\label{ncond0}
    |M_\ell| + d_\ell \leq \rho_\ell \quad \text{ for }\ell\in[k].
\end{align}
As $|N_\ell^i| \geq \rounddown{\rho_\ell/n}$ for $i\in[n]$, there are at least $(n-r)\rounddown{\rho_\ell/n}$ occurrences of $\ell$ in the $n-r$ rows of $N$ outside of $M$. Moreover, as $|N_\ell^i|\leq \roundup{\rho_\ell/n}$ (and so $|^iN_\ell|\leq \roundup{\rho_\ell/n}$), there are at most $(n-r)\roundup{\rho_\ell/n}$ occurrences of $\ell$ in the $n-r$ rows of $N$ outside of $M$ and in the $n-r$ columns of $N$ outside of $M$. As the $d_\ell$ occurrences of $\ell$ on the diagonal tail of $N$ outside of $M$ are included in both the remaining $n-r$ rows and in the remaining $n-r$ columns, we must have the following.
\begin{align}\label{ncond1}
    (n-r)\rounddown{\dfrac{\rho_\ell}{n}} \leq \rho_\ell - |M_\ell| \leq 2(n-r)\roundup{\dfrac{\rho_\ell}{n}} - d_\ell \quad \text{ for }\ell\in[k].
\end{align}

Since for $\ell\in[k]$, the number of occurrences of $\ell$ off the diagonal outside of $M$ must be even, we also have that
\begin{align} \label{ncond2}
    \Big| \{\ell\in[k] \ : \ \rho_\ell - |M_\ell| - d_\ell\not\equiv  0 \Mod 2\}\Big| \leq n-r-\sum_{\ell\in[k]}d_\ell.
\end{align}
If for some $\ell\in[k]$ we have that $\rho_\ell - |M_\ell| - d_\ell$ is odd, then 
\begin{align*}
    \Big|\{ N(i,i) \ : \ N(i,i)=\ell\}\Big| \geq \Big|\{ M(i,i) \ : \ M(i,i)=\ell\}\Big| + d_\ell + 1,
\end{align*}
so we also have that
\begin{align*} 
    n-r-\sum_{\ell\in[k]}d_\ell \equiv \sum_{\ell\in[k]}(\rho_\ell - |M_\ell| - d_\ell) \equiv \Big| \{\ell\in[k] \ : \ \rho_\ell - |M_\ell| - d_\ell \text{ is odd}\}\Big| \Mod 2.
\end{align*}
Hence, the following is also necessary.
\begin{align} \label{ncond3}
    n-r-\sum_{\ell\in[k]}d_\ell \equiv \Big| \{\ell\in[k] \ : \ \rho_\ell - |M_\ell| - d_\ell \text{ is odd}\}\Big| \Mod 2.
\end{align}

\noindent Given a prescribed diagonal tail $\vb*{d}$, we say that an $r\times r$ externally symmetric equitable $\brho$-Latin rectangle $M$ is {\it equitably $(\brho, \vb*{d})$-admissible} if it satisfies \eqref{ncond0}---\eqref{ncond3}. Here is our first main result.

\begin{theorem} \label{mainthm}
    Let $r\in\mathbb{N}\cup\{0\}$ and $k,n\in\mathbb{N}$ such that $r\leq n$ and $k\leq n^2$. Let $\brho=(\rho_1,\dots,\rho_k)$ such that $\rho_\ell\in\mathbb{N}$ for $\ell\in[k]$ and $\sum_{\ell\in[k]}\rho_\ell = n^2$, and let $\vb*{d}=(d_1,\dots,d_k)$ such that $d_\ell\in\mathbb{N}\cup\{0\}$ for $\ell\in[k]$ and $\sum_{\ell\in[k]}d_\ell \leq n-r$. Then an $r\times r$ externally symmetric equitable $\brho$-Latin rectangle $M$ on $[k]$ can be extended to an $n\times n$ equitable $\brho$-Latin square on $[k]$ which is symmetric off $M$ and has the prescribed diagonal tail $\vb*{d}$ if and only if $M$ is equitably $(\brho, \vb*{d})$-admissible and the following hold, where $Y=\{\ell\in[k] \ : \ n\nmid \rho_\ell\}$.
\begin{align}
    &|I|\Big(\sum_{\ell\in[k]}\roundup{\frac{\rho_\ell}{n}}-n\Big) \leq \sum_{\ell\in Y}\min\Big\{n\roundup{\dfrac{\rho_\ell}{n}} - \rho_\ell, \eta_I(\ell)\Big\}    &&\text{ for } I\subseteq[r],\label{ncond4}\\
    &\sum_{\ell\in K}\max\Big\{r\roundup{\frac{\rho_\ell}{n}} - \rounddown{\frac{\rho_\ell + |M_\ell| - d_\ell}{2}}, r\left(\roundup{\frac{\rho_\ell}{n}} - \rounddown{\frac{\rho_\ell}{n}}\right) + n\rounddown{\frac{\rho_\ell}{n}} - \rho_\ell,0\Big\}\label{ncond5}\\
     &\quad\quad\quad \leq \sum_{i\in [r]} \min\Big\{\sum_{\ell\in[k]}\roundup{\frac{\rho_\ell}{n}} - n, \eta_K(i)\Big\}  &&\text{ for } K\subseteq Y.\nonumber
\end{align}
\end{theorem}

Our main theorem unifies several classical results from the aforementioned literature. In particular, when \(M\) is symmetric, \(k=n\), \(\rho_\ell=n\) for  \(\ell\in[n]\), and \(\sum_{\ell\in[n]} d_\ell=0\), Theorem~\ref{mainthm} reduces to Cruse's theorem \cite{MR329925}. When \(M\) is symmetric, \(k=n\), \(\rho_\ell=n\) for  \(\ell\in[n]\), and \(\sum_{\ell\in[n]} d_\ell=n-r\), it yields Hoffman's theorem \cite{MR694466}. More generally, when \(k=n\) and \(\rho_\ell=n\) for  \(\ell\in[n]\), Theorem~\ref{mainthm} recovers Andersen's theorem \cite{MR772580}. Finally, when \(M\) is symmetric and \(\rho_\ell\le n\) for  \(\ell\in[k]\), Theorem~\ref{mainthm} implies the theorem of Bahmanian and Hilton \cite{MR5024856}.

Baranyai's theorem constructs almost regular colorings of complete uniform hypergraphs \cite{MR0416986}. From this perspective, our result may be viewed as a generalization of Baranyai's theorem to the setting of complete graphs. More precisely, we determine necessary and sufficient conditions under which a partial coloring of \(\mathbb{K}_r\), the complete graph \(K_r\) with a loop added at each vertex, can be extended to an almost regular coloring of \(\mathbb{K}_n\).

The remainder of our paper is organized as follows. In Section \ref{AHs:prereq}, we provide further prerequisites to the proofs of our main result. We prove Theorem \ref{mainthm} in Section \ref{s:AndHoff}, provide additional corollaries of our main results in Section \ref{s:cor}, and conclude the paper with several open problems in Section \ref{s:openp}.

\section{Prerequisites} \label{AHs:prereq}

In this paper, we say that for $a,b\in\mathbb{R}$, $a \approx b$ if $\rounddown{b}\leq a\leq \roundup{b}$. Note that this is a transitive relation, but is not symmetric. 
Furthermore, if $a \approx b$, then $a/n \approx b/n$ for $n\in\mathbb{N}$. 
We denote by $\mult_G(e)$ the multiplicity of an edge $e$ in a graph $G$, where $e$ may be a set of either one or two vertices in $G$. An {\it $i$-loop} is an edge that is incident with only one vertex and contributes $i$ to the degree of that vertex. Given an edge $e$ incident with a vertex $u$, we write $e=u$ or $e=u^2$ if $e$ is a 1-loop or a 2-loop, respectively.  Let $\mult_G(vX) = \sum_{x\in X}\mult_G(vx)$ for $v\in V(G), X\subseteq V(G)$. The set of 1-loops of a graph $G$ is denoted by $E^1(G)$, and we define $E^2(G)$ to be the remaining edges in $E(G)\backslash E^1(G)$, which all contain either two distinct vertices or two copies of one vertex.

We denote by $\mathbb{K}_n$ the complete graph on $n$ vertices in which each vertex is additionally incident to a 1-loop. Note that $\mult_{\mathbb{K}_n}(u)=\mult_{\mathbb{K}_n}(uv)=1$ for $u,v \in V(\mathbb{K}_n)$ with $u\neq v$. Furthermore, there is a one-to-one correspondence between $n\times n$ symmetric Latin squares and 1-factorizations of $\mathbb{K}_n$.

Let $G$ be a graph, let $\alpha\in V(G)$, and suppose the edges of $G$ are colored with $k$ colors. We denote by $G(\ell)$ the spanning subgraph of $G$ induced by edges of color $\ell$; that is, $G(\ell)$ is the subgraph of $G$ with vertex set $V(G)$ and edge set precisely the set of all edges in $G$ with color $\ell$. We may {\it split} $\alpha$ into $p$ new vertices $\alpha_1,\dots,\alpha_p$ and obtain a new graph $F$ in which the edges incident with $\alpha$ in $G$ are shared among the new vertices $\alpha_1,\dots,\alpha_p$ in $V(F)$ so that each edge $\alpha u$ in $G$ corresponds to an edge $\alpha_i u$ in $F$ for $u\in V(F)$ and some $i\in[p]$, each $2$-loop on the vertex $\alpha$ corresponds to an edge $\alpha_i\alpha_j$ for some $i,j\in[p]$, and each $1$-loop on the vertex $\alpha$ corresponds to a $1$-loop on a vertex $\alpha_i$ for some $i\in[p]$. Then we say that $F$ is a {\it detachment} of $G$. The reverse process is also possible. If we {\it amalgamate} the $p$ vertices $\alpha_1,\dots,\alpha_p$ in $F$ into a single vertex $\alpha$, we may obtain a new graph $G$ in which each $\alpha_i u$ edge in $F$ corresponds to a copy of the edge $\alpha u$ in the resulting graph $G$ for $i\in[p]$, each edge between two vertices $\alpha_i, \alpha_j$ corresponds to a $2$-loop on $\alpha$, and each $1$-loop on a vertex $\alpha_i$ corresponds to a $1$-loop on $\alpha$. In this case, we call $G$ an {\it amalgamation} of $F$.
We use the following detachment lemma in the proof of our main result.

\begin{lemma} \label{amalgambahcpc}\cite[Theorem 4.1]{MR2942724}
  Let $G$ be a graph whose edges are colored with $k$ colors, and let $\alpha\in V(G)$. There exists a graph $F$ obtained by splitting $\alpha$ into $\alpha_1,\dots,\alpha_p$ such that
 \begin{enumerate}
     \item [\textup{(i)}] $\dg_{F(\ell)}(\alpha_i)\approx \dg_{G(\ell)}(\alpha)/p$  for  $i\in [p],\ell\in [k]$;
     \item [\textup{(ii)}] $\mult_F(\alpha_i)\approx \mult_G(\alpha)/p$ for  $i\in[p]$;
     \item [\textup{(iii)}] $\mult_F(\alpha_i u)\approx \mult_G(\alpha u)/p$ for  $i\in[p],u\in V(G)\backslash \{\alpha\}$;
     \item  [\textup{(iv)}]  $\mult_F(\alpha_i \alpha_j)\approx \mult_G(\alpha^2)/\binom{p}{2}$ for  $i,j\in[p]$ with $i\neq j$.
 \end{enumerate} 
\end{lemma}

We introduce one more theorem that is used in the proof of our main result. Let $f$ and $g$ be integer functions on the vertex set of a graph $G$ with $0\leq g(x)\leq f(x)$ for $x$. A {\it $(g,f)$-factor} is defined to be a spanning subgraph $F$ of $G$ such that $g(x)\leq \dg_F(x)\leq f(x)$ for each $x$. We define $N_G(A)$ to be the  neighborhood of $A$ in $G$. The following is a special case of Lov\'{a}sz's  $(g,f)$-factor Theorem \cite{MR325464}.

\begin{theorem} \cite[Theorem 1]{MR1081839}\label{gffactorthm1}
A given bipartite graph $G[X,Y]$ has a $(g,f)$-factor if and only if the following holds.
\begin{align*}
\sum_{a\in A}f(a)&\geq  \sum_{a\notin A} \max\left\{ g(a) - \dg_{G-A}(a),0 \right\} \quad &&\text{ for } A\subseteq X\cup Y.
\end{align*} 
\end{theorem}

Theorem \ref{gffactorthm1} was later simplified as follows.

\begin{theorem} \cite[Theorem 5]{MR3564794}\label{gffactorthm2}
A given bipartite graph $G[X,Y]$ has a $(g,f)$-factor if and only if the following holds.
\begin{align*}
    \sum_{a\in A} g(a)&\leq \sum_{a\in N_G(A)} \min \left\{ f(a), \mult_G(aA)\right\}
\quad &&\text{ for } A\subseteq X, A\subseteq Y.
\end{align*} 
\end{theorem}

\section{Extending Externally Symmetric Equitable \texorpdfstring{$\brho$}{rho}-Latin Rectangles} \label{s:AndHoff}

In this section, we prove Theorem \ref{mainthm}.

\begin{proof}
Let $M$ be an $r\times r$ externally symmetric equitable $\brho$-Latin rectangle and fix a prescribed diagonal tail $\vb*{d}$. 
We have established the necessity of \eqref{ncond0}, \eqref{ncond1}, \eqref{ncond2}, and \eqref{ncond3} in the introduction. 
The necessity of the remaining conditions will become clear at the end of the proof. 

To prove the sufficiency, assume that $M$ is equitably $(\brho, \vb*{d})$-admissible, so \eqref{ncond0}--\eqref{ncond3} hold.
Let $F=\mathbb K_n \backslash \mathbb{K}_r$, the graph resulting from removing the edges of a complete subgraph with loops $\mathbb{K}_r$ from the complete graph with loops $\mathbb{K}_n$. Then $V(F)=X^*:=\{x_1,\dots,x_n\}$ and $E(F) = \{x_ix_j \ : \ i\geq r+1 \text{ or } j\geq r+1 \text{ or both}\}$. Note that there is one edge in $F$ for each pair of off-diagonal cells $(i,j)$ and $(j,i)$ and one 1-loop in $F$ for each diagonal cell $(i,i)$ that must be added to the array $M$ in order to extend it to an $n\times n$ array. Let $X=\{x_1,\dots, x_r\}$.
For $\ell\in[k]$, we color $d_\ell$ arbitrary uncolored 1-loops (incident with vertices of $X^* \backslash X$) with color $\ell$. These colored 1-loops represent the prescribed diagonal of the desired equitable $\brho$-Latin square. Then
\begin{align*}
    \sum_{i\in[n]\backslash[r]}\mult_{F(\ell)}(x_i)=d_\ell \text{ for }\ell\in[k].
\end{align*}
Let $G$ be the graph obtained by amalgamating $x_{r+1},\dots,x_n$ of $F$ into a single vertex $\alpha$. Then for $\ell\in[k],i\in[r]$, we have that
\begin{align*}
    \mult_G(\alpha)=n-r, &&\mult_G(x_i\alpha)=n-r, &&\mult_G(\alpha^2) = \binom{n-r}{2},&&\mult_{G(\ell)}(\alpha)=d_\ell.
\end{align*}
First, we note that we can extend $M$ to an $n\times n$ equitable $\brho$-Latin square $N$ with the prescribed diagonal tail $\vb*{d}$ which is symmetric off $M$ if and only if we can color all edges in $F$ such that the following is satisfied.
\begin{align} \label{AHFconditions}
\begin{cases}
    |M_\ell^i| + \dg_{F(\ell)}(x_i) \approx \dfrac{\rho_\ell}{n} &\text{ for } i\in[r], \ell\in[k],\\
    \dg_{F(\ell)}(x_i) \approx \dfrac{\rho_\ell}{n} &\text{ for } i\in[n]\backslash[r], \ell\in[k],\\
    |M_\ell| + |E^1(F(\ell))| + 2|E^2(F(\ell))| = \rho_\ell &\text{ for } \ell\in[k].
\end{cases}
\end{align}
To see this, first observe that if $N$ is an $n\times n$ equitable $\brho$-Latin square with prescribed diagonal tail $\vb*{d}$ which contains a copy of $M$ in its top left $r\times r$ subarray and is symmetric off this subarray, then we obtain a coloring of $F$ meeting the above conditions in the following way. We color the edge $x_ix_j$ of $F$ with color $\ell$ if cell $(i,j)$ and cell $(j,i)$ contain the symbol $\ell$ in $N$ for $i,j\in[n]\backslash[r]$, and we color the $1$-loop on vertex $x_i$ of $F$ with color $\ell$ if cell $(i,i)$ contains the symbol $\ell$ in $N$ for $i\in[n]\backslash [r]$. This is possible because $N$ is symmetric off $M$. As $N$ has has the prescribed diagonal tail $\vb*{d}$, we will have that $\sum_{[n]\backslash[r]}\mult_{F(\ell)}(x_i) = d_\ell$ for $\ell\in[k]$ in the resulting coloring of $F$. It is straightforward to verify that the three conditions in \eqref{AHFconditions} are also met by coloring the edges of $F$ in this way. 

\noindent Conversely, if the edges of $F$ are colored so that \eqref{AHFconditions} holds, we may fill the cells outside the top left subarray of $N$ by placing the symbol $\ell$ in cells $(i,j)$ and $(j,i)$ of $N$ whenever $x_ix_j$ is colored with color $\ell$ in the coloring of $F$, and we may fill the remaining cells in the top left $r\times r$ subarray of $N$ so that this subarray is a copy of $M$. The resulting array $N$ will be symmetric off $M$, and by \eqref{AHFconditions}, we will also have that $|N_\ell| = \rho_\ell$ for $\ell\in[k]$. Moreover, as $M$ is externally symmetric, we will have that $|N_\ell^i|\approx \rho_\ell/n$ and $|^iN_\ell|\approx \rho_\ell/n$ for $i\in[n], \ell\in[k]$ by \eqref{AHFconditions}. Thus, $N$ will be an equitable $\brho$-Latin square. 
As there are exactly $d_\ell$ $1$-loops of color $\ell$ in the graph $F$ both before and after its coloring is completed to satisfy \eqref{AHFconditions}, the resulting array $N$ will also have the prescribed diagonal tail $\vb*{d}$.\\

\noindent {\bf Claim 1. }
We can color the edges of $F$ such that \eqref{AHFconditions} is satisfied if and only if we can color the edges of $G$, an amalgamation of $F$, such that the following hold. 
\begin{align}\label{degeqG}
    \begin{cases}
    |M_\ell^i| + \dg_{G(\ell)}(x_i) \approx \dfrac{\rho_\ell}{n} &\text{ for } i\in [r], \ell\in[k],\\
    \dfrac{\dg_{G(\ell)}(\alpha)}{n-r} \approx \dfrac{\rho_\ell}{n} &\text{ for } \ell\in[k],\\
    |M_\ell| + |E^1(G(\ell))|+2|E^2(G(\ell))|= \rho_\ell &\text{ for } \ell\in[k].
    \end{cases}
\end{align}
To prove this claim, first suppose that the coloring of $F$ is extended such that \eqref{AHFconditions} holds. Then a coloring of $G$ satisfying \eqref{degeqG} can be found by amalgamating the $n-r$ vertices $x_{r+1},\dots,x_n$ in $V(F)$ into a single vertex $\alpha$.

\noindent Conversely, suppose that the coloring of $G$ can be extended such that \eqref{degeqG} holds.
Then by Lemma \ref{amalgambahcpc}, we may detach $\alpha$ into $n-r$ vertices $\alpha_1,\dots,\alpha_{n-r}$ and obtain a graph $F'$ for which
\begin{align*}
    &\mult_{F'}(\alpha_i)= \dfrac{\mult_G(\alpha)}{n-r}=1 &&\text{ for }i\in[n-r],\\
    &\mult_{F'}(\alpha_i x_j)= \dfrac{\mult_G(\alpha x_j)}{n-r} = 1 &&\text{ for }i\in[n-r],j\in [r],\\
    &\mult_{F'}(\alpha_i \alpha_j)= \dfrac{\mult_G(\alpha^2)}{\binom{n-r}{2}} = 1&&\text{ for }i,j\in[n-r] \text{ with }i\neq j,\\
    &\dg_{F'(\ell)}(\alpha_i)\approx \dfrac{\dg_{G(\ell)}(\alpha)}{n-r} \approx \dfrac{\rho_\ell}{n} &&\text{ for }i\in[n-r],\ell\in[k]. 
\end{align*}
Hence, $F'\cong K_n$ and $F'$ satisfies \eqref{AHFconditions}, so this completes the proof of Claim 1.\\ 

Let $\Gamma[X,[k]]$ be the bipartite graph whose edge multi-set contains $\roundup{\rho_\ell/n}-|M_\ell^i|$ copies of the edge $x_i\ell$ for $i\in[r], \ell\in[k]$. The edges in $\Gamma$ correspond to the maximum number of occurrences of $\ell\in[k]$ we could add to each row of the given equitable $\brho$-Latin rectangle $M$ without exceeding $\roundup{\rho_\ell/n}$ occurrences of $\ell$ in a row of an equitable $\brho$-Latin square. Note that $\Gamma$ satisfies the following.
\begin{align}\label{degeqGamma}
 \begin{cases}
\dg_\Gamma(x_i)= \displaystyle \sum_{\ell\in[k]} \roundup{\dfrac{\rho_\ell}{n}} - r & \text{ for } i \in [r],\vspace{1mm}\\
\dg_{\Gamma}(\ell)=r\roundup{\dfrac{\rho_\ell}{n}} -|M_\ell|& \text{ for  } \ell \in [k],\vspace{2mm}\\
\mult_\Gamma(x_i\ell) =\roundup{\dfrac{\rho_\ell}{n}} - |M_\ell^i| & \text{ for } i\in[r],\ell\in[k].
\end{cases}
\end{align}
To see that such a graph is well-defined, note that by definition of the equitable $\brho$-Latin rectangle $M$, we have that $\roundup{\rho_\ell/n}\geq |M_\ell^i|$ and $r\roundup{\rho_\ell/n}\geq |M_\ell|$, and note that by definition of the $k$-tuple $\brho$, we have that $\sum_{\ell\in[k]}\roundup{\rho_\ell/n} \geq \sum_{\ell\in[k]}\rho_\ell/n = n \geq r$.\\

\noindent {\bf Claim 2. }
We can color the edges of $G$ so that \eqref{degeqG} is satisfied if and only if there exists a subgraph $\Theta$ of $\Gamma$ with $r(n-r)$ edges so that
\begin{equation} \label{AHThetaconditions}
    \begin{cases}
        \dg_{\Theta}(x_i) = n-r & \text{ for } i\in[r],\vspace{1mm}\\
        \dfrac{\rho_\ell - |M_\ell| - \dg_\Theta(\ell)}{n-r} \approx \dfrac{\rho_\ell}{n} & \text{ for } \ell\in[k],\vspace{2mm}\\
        \dg_{\Theta}(\ell) \leq \rounddown{\dfrac{\rho_\ell - |M_\ell| - d_\ell}{2}} &\text{ for } \ell\in [k],\vspace{1mm}\\
        \mult_\Theta(x_i\ell) + |M_\ell^i| \approx \dfrac{\rho_\ell}{n} &\text{ for } i\in[r],\ell\in[k].
    \end{cases}
\end{equation}
As $\sum_{\ell\in[k]}\rho_\ell = n^2$, we have that $\dg_\Gamma(x_i)\geq n-r$. 
Note that \eqref{AHThetaconditions} leads to one lower bound and two upper bounds on $\dg_\Theta(\ell)$ for $\ell\in[k]$, stated below.
\begin{align*}
    \rho_\ell - |M_\ell| - (n-r)\roundup{\dfrac{\rho_\ell}{n}} \leq \dg_\Theta(\ell) \leq \min\Big\{ \rho_\ell - |M_\ell| - (n-r)\rounddown{\dfrac{\rho_\ell}{n}}, \rounddown{\dfrac{\rho_\ell - |M_\ell| - d_\ell}{2}}\Big\}.
\end{align*}
Since $\rho_\ell\leq n\roundup{\rho_\ell/n}$, the lower bound on $\dg_\Theta(\ell)$ in the above equation is not greater than $\dg_\Gamma(\ell)$, as shown by the following.
\begin{align*}
    \dg_\Gamma(\ell) = r\roundup{\dfrac{\rho_\ell}{n}} - |M_\ell| \geq r\roundup{\dfrac{\rho_\ell}{n}} - |M_\ell| + \rho_\ell - n\roundup{\dfrac{\rho_\ell}{n}}.
\end{align*}
Moreover, we have that $\rho_\ell - |M_\ell| - (n-r)\rounddown{\rho_\ell/n} \geq 0$ by \eqref{ncond1} and that $\rho_\ell - |M_\ell| - d_\ell \geq 0$ by \eqref{ncond0}. Hence, the upper bounds on $\dg_\Theta(\ell)$ in \eqref{AHThetaconditions} are both non-negative.
Finally, observe that for $\ell\in[k]$, we have by \eqref{ncond1} that
$(n-r)\roundup{\rho_\ell/n} \geq (\rho_\ell - |M_\ell| + d_\ell)/2$,
and it follows that
\[
\frac{\rho_\ell-|M_\ell|-d_\ell}{2}
=\rho_\ell-|M_\ell|
-\frac{\rho_\ell-|M_\ell|+d_\ell}{2}
\ge \rho_\ell-|M_\ell|
-(n-r)\roundup{\dfrac{\rho_\ell}{n}}.
\]

Since $\rho_\ell - |M_\ell| - (n-r)\roundup{\rho_\ell/n}$ is an integer, it follows that $\rho_\ell - |M_\ell| - (n-r)\roundup{\rho_\ell/n} \leq \rounddown{(\rho_\ell - |M_\ell| - d_\ell)/2}$ for $\ell\in[k]$, and so the upper bounds on $\dg_\Theta(\ell)$ are both greater than or equal to the lower bound on $\dg_\Theta(\ell)$.

\noindent To prove Claim 2, suppose that the coloring of $G$ is extended such that  \eqref{degeqG} is satisfied. Define  $\Theta [X, [k]]\subseteq \Gamma$ to be the bipartite graph with vertex set $X\cup[k]$ and an edge multi-set in which there are $\dg_{G(\ell)}(x_i)$ copies of the edge $x_i\ell$ for $i\in[r], \ell\in[k]$. 
Then for $i\in[r]$,
\begin{align*}
    \dg_\Theta(x_i)=\mult_G(x_i\alpha)=n-r.
\end{align*}
Let $\ell\in[k]$. We have that 
\begin{align*}
    \rho_\ell &=|M_\ell| + |E^1(G(\ell))| + 2|E^2(G(\ell))|\\
    &= |M_\ell| + \mult_{G(\ell)}(\alpha) + 2\sum_{i\in[r]}\mult_{G(\ell)}(\alpha x_i) + 2\mult_{G(\ell)}(\alpha^2).
\end{align*}
Hence, since $\dg_\Theta(\ell) = \sum_{i\in[r]}\mult_{G(\ell)}(x_i\alpha)$, 
\begin{align*}
    \rho_\ell - |M_\ell| - \dg_\Theta(\ell)&= \rho_\ell - |M_\ell| - \sum_{i\in[r]}\mult_{G(\ell)}(x_i\alpha)\\
    &= \mult_{G(\ell)}(\alpha) + \sum_{i\in[r]}\mult_{G(\ell)}(x_i\alpha) + 2\mult_{G(\ell)}(\alpha^2)\\
    &= \dg_{G(\ell)}(\alpha),
\end{align*}
and it follows that
\begin{align*}
    \frac{\rho_\ell - |M_\ell| - \dg_\Theta(\ell)}{n-r} &= \frac{\dg_{G(\ell)}(\alpha)}{n-r}\approx \frac{\rho_\ell}{n}.
\end{align*}
Since
\begin{align*}
    \rho_\ell&=|M_\ell| + |E^1(G(\ell))|+2|E^2(G(\ell))|\\
        &=|M_\ell|+\mult_{G(\ell)}(\alpha)+2\mult_{G(\ell)}(\alpha X)+2\mult_{G(\ell)}(\alpha^2)\\
        &\geq |M_\ell|+d_\ell+2\dg_\Theta(\ell),
\end{align*}
we have that $\dg_\Theta(\ell)\leq \rounddown{(\rho_\ell - |M_\ell| - d_\ell)/2}$. 
Finally,
\begin{align*}
    \mult_\Theta(x_i\ell) + |M_\ell^i| = \dg_{G(\ell)}(x_i) + |M_\ell^i| \approx\frac{\rho_\ell}{n},
\end{align*}
so \eqref{AHThetaconditions} is satisfied.

\noindent Conversely, suppose that $\Theta\subseteq \Gamma$ satisfying \eqref{AHThetaconditions} exists. 
For $\ell\in [k]$, if $\ell x_i \in E(\Theta)$ for some $i\in [r]$, we color $\mult_\Theta(\ell x_i)$ distinct $x_i\alpha$-edges in $G$ with $\ell$. Since $\dg_\Theta(x_i)=n-r$ for $i\in [r]$, all the edges between $\alpha$ and $X$ can be colored this way.
Since $\mult_\Theta(x_i\ell) + |M_\ell^i| \approx \rho_\ell/n$, 
we have that $d_{G(\ell)}(x_i) + |M_\ell^i| \approx \rho_\ell/n$ for $\ell \in [k]$, $i\in[r]$. 

\noindent Let $O\subseteq[k]$ be the set of colors for which $\rho_\ell - |M_\ell| - d_\ell \equiv 1 \Mod 2$.
Then by \eqref{ncond2} and \eqref{ncond3}, $(n-r-|O|-\sum_{\ell\in[k]}d_\ell)/2$ is a non-negative integer, and 
by \eqref{AHThetaconditions}, $\rounddown{(\rho_\ell - |M_\ell| - d_\ell)/2} - \dg_\Theta(\ell) \geq 0$ for $\ell\in[k]$. Hence, $(\rho_\ell - |M_\ell| - d_\ell)/2 - \dg_\Theta(\ell) \geq 0$ for $\ell\in[k]\backslash O$, and $(\rho_\ell - |M_\ell| - d_\ell-1)/2 - \dg_\Theta(\ell) \geq 0$ for $\ell\in O$.
Furthermore, 
\begin{align*}
    n-r-|O|-\sum_{\ell\in[k]}d_\ell &\leq (n-r)^2 - |O| - \sum_{\ell\in[k]}d_\ell\\
    &=n^2 - r^2 - \sum_{\ell\in[k]}d_\ell - 2r(n-r) - |O|\\
    &= \sum_{\ell\in[k]}(\rho_\ell - |M_\ell| - d_\ell) - 2\sum_{\ell\in[k]}\dg_\Theta(\ell) - |O|\\
    &= \sum_{\ell\in[k]\backslash O}(\rho_\ell - |M_\ell| - d_\ell - 2\dg_\Theta(\ell)) + \sum_{\ell\in O}(\rho_\ell - |M_\ell| - d_\ell - 1 - 2\dg_\Theta(\ell)\Big).
\end{align*}
Thus, there exists a sequence of integers $b_1,\dots,b_k$ such that
\begin{align*}
\begin{cases}
    \displaystyle \sum_{\ell\in[k]}b_\ell = \dfrac{1}{2}\Big(n-r-|O|-\sum_{\ell\in[k]}d_\ell\Big), \quad&\\
    0\leq b_\ell \leq \dfrac{1}{2}\left(\rho_\ell - |M_\ell| - d_\ell\right) - \dg_\Theta(\ell) \quad& \text{ for } \ell\in[k]\backslash O,\vspace{1mm}\\
    0\leq b_\ell \leq \dfrac{1}{2}\left(\rho_\ell - |M_\ell| - d_\ell - 1\right) - \dg_\Theta(\ell) \quad& \text{ for } \ell\in O.
\end{cases}
\end{align*}
We define $\overline{d_\ell}$ as follows.
\begin{align*}
    \overline{d_\ell} = \begin{cases}
        2b_\ell & \text{ for } \ell\in[k]\backslash O,\\
        2b_\ell + 1 & \text{ for } \ell\in O.
    \end{cases}
\end{align*}
Thus, we have that $\overline{d_1},\dots,\overline{d_k}$ satisfies the following.
\begin{align} \label{bard}
\begin{cases}
    \displaystyle \sum_{\ell\in[k]}\overline{d_\ell} = n-r-\sum_{\ell\in[k]}d_\ell, \quad &\vspace{1mm}\\
    \overline{d_\ell} \equiv \rho_\ell - |M_\ell| - d_\ell \Mod 2 \quad &\text{ for } \ell\in[k],\vspace{1mm}\\
    \dg_\Theta(\ell) \leq \dfrac{1}{2}(\rho_\ell - |M_\ell| - d_\ell - \overline{d_\ell}) \quad &\text{ for } \ell\in[k].
\end{cases}
\end{align}
We now color $1$-loops and $2$-loops of $G$ in the following way. We color $\overline{d_\ell}$ 1-loops with color $\ell$ such that $\mult_{G(\ell)}(\alpha) = d_\ell + \overline{d_\ell}$ and so that there are then $\mult_{G(\ell)}(\alpha^2) = \frac{1}{2}(\rho_\ell - |M_\ell| - d_\ell - \overline{d_\ell}) - \dg_\Theta(\ell)$ remaining $2$-loops to color with color $\ell$ for $\ell\in[k]$. This is possible by definition of $\overline{d_\ell}$ and because
\begin{align*}
    \sum_{\ell\in[k]}(\rho_\ell - |M_\ell| - 2\dg_\Theta(\ell) - d_\ell - \overline{d_\ell})
    &= n^2 - r^2 - 2r(n-r) - (n-r)\\
    &= 2\binom{n-r}{2}= 2\mult_G(\alpha^2).
\end{align*}
Now we have that for $\ell\in[k]$,
\begin{align*}
    |M_\ell| + |E^1(G(\ell))| + 2|E^2(G(\ell))| = |M_\ell| + 2\dg_\Theta(\ell) + \mult_{G(\ell)}(\alpha) + 2\mult_{G(\ell)}(\alpha^2) = \rho_\ell.
\end{align*}
Finally, for $\ell\in[k]$,
  \begin{align*}
\frac{\dg_{G(\ell)}(\alpha)}{n-r} &=\frac{1}{n-r}\left(\mult_{G(\ell)}(\alpha X)+ \mult_{G(\ell)}(\alpha)+ 2\mult_{G(\ell)}(\alpha^2)\right) \\
 &= \frac{1}{n-r}\left(\dg_\Theta(\ell)+d_\ell+\overline{d_\ell} + (\rho_\ell-|M_\ell|-d_\ell-\overline{d_\ell}-2\dg_\Theta(\ell))\right)\\
 &= \frac{1}{n-r}\left(\rho_\ell-|M_\ell|-\dg_\Theta(\ell)\right)\\
 &\approx \frac{\rho_\ell}{n} \text{ by \eqref{AHThetaconditions}},
\end{align*}
so \eqref{degeqG} holds. This completes the proof of Claim 2.\\

Let $\tilde{\Gamma}[X,[k]]$ be the simple bipartite subgraph of $\Gamma$ whose edge set consists of exactly one edge $x_i\ell$ for $i\in[r]$ and each free color $\ell$ such that $|M_\ell^i| <\roundup{\rho_\ell/n}$. In other words,
\begin{align*}
    E(\tilde{\Gamma}) = \Big\{ x_i\ell \ : \ i\in [r], \ell\in[k], |M_\ell^i| < \roundup{\frac{\rho_\ell}{n}}, n\nmid\rho_\ell\Big\}.
\end{align*}
The edges in the bigraph $\tilde{\Gamma}$ indicate whether a free symbol $\ell\in[k]$ will occur a total of $\roundup{\rho_\ell/n}$ or $\rounddown{\rho_\ell/n}$ times in each row of the equitable $\brho$-Latin square to which $M$ is being extended. Since each forced symbol must occur exactly $\rho_\ell/n$ times in every row of an equitable $\brho$-Latin square, $\tilde{\Gamma}$ does not include any edges $x_i\ell$ for forced colors $\ell\in[k]$.\\

\noindent {\bf Claim 3. }
There exists a subgraph $\Theta\subseteq\Gamma$ satisfying \eqref{AHThetaconditions} if and only if there exists $\tilde{\Theta}\subseteq\tilde{\Gamma}$ for which the following hold.
\begin{align} \label{AHtildeThetaconditions}
    \begin{cases}
        \dg_{\tilde{\Theta}}(x_i) = \displaystyle \sum_{\ell\in[k]}\roundup{\frac{\rho_\ell}{n}} - n &\text{ for } i\in[r],\vspace{1mm}\\
        \dg_{\tilde{\Theta}}(\ell) \leq n\roundup{\dfrac{\rho_\ell}{n}} - \rho_\ell &\text{ for }\ell\in[k], \vspace{1mm}\\
        \dg_{\tilde{\Theta}}(\ell) \geq n\rounddown{\dfrac{\rho_\ell}{n}}  + r\left(\roundup{\dfrac{\rho_\ell}{n}} - \rounddown{\dfrac{\rho_\ell}{n}}\right) - \rho_\ell &\text{ for }\ell\in[k], \vspace{1mm}\\
        \dg_{\tilde{\Theta}}(\ell) \geq r\roundup{\dfrac{\rho_\ell}{n}} - \rounddown{\dfrac{\rho_\ell + |M_\ell| - d_\ell}{2}} &\text{ for }\ell\in [k].
    \end{cases}
\end{align}
Note that just as \eqref{AHThetaconditions} results in one lower bound and two upper bounds on $\dg_\Theta(\ell)$, we have that \eqref{AHtildeThetaconditions} results in one upper bound and two lower bounds on $\dg_{\tilde{\Theta}}(\ell)$.

\noindent To prove this claim, first suppose that $\Theta\subseteq\Gamma$ satisfying \eqref{AHThetaconditions} exists. Let $\tilde{\Theta}$ be the graph with vertex set $X\cup [k]$ and edge set $E(\Gamma)\backslash E(\Theta)$. Then $\mult_{\tilde{\Theta}}(x_i\ell) \leq \roundup{\rho_\ell/n}-\rounddown{\rho_\ell/n}$ for $i\in[r],\ell\in[k]$, so $\mult_{\tilde{\Theta}}(x_i\ell)=0$ when $n\mid \rho_\ell$,  and $\mult_{\tilde{\Theta}}(x_i\ell)\leq 1$ when $n \nmid \rho_\ell$. Hence, $\tilde{\Theta}$ is a subgraph of $\tilde{\Gamma}$. Moreover, for $i\in[r]$ and $\ell\in[k]$, we have that 
\begin{align*}
    \dg_{\tilde{\Theta}}(x_i) &= \dg_\Gamma(x_i) - \dg_\Theta(x_i)= \sum_{\ell\in[k]}\roundup{\dfrac{\rho_\ell}{n}} - r - \left(n-r\right)= \sum_{\ell\in[k]}\roundup{\dfrac{\rho_\ell}{n}} - n,
\end{align*}
and we have that
\begin{align*}
    \dg_{\tilde{\Theta}}(\ell) &= \dg_\Gamma(\ell) - \dg_\Theta(\ell)= r\roundup{\dfrac{\rho_\ell}{n}} - |M_\ell| - \dg_\Theta(\ell).
\end{align*}
Hence, by \eqref{AHThetaconditions}, we have for $i\in[r]$ and $\ell\in[k]$ that
\begin{align*}
    \dg_{\tilde{\Theta}}(\ell) &\leq r\roundup{\dfrac{\rho_\ell}{n}} - |M_\ell| - \left(\rho_\ell - |M_\ell| - (n-r)\rounddown{\dfrac{\rho_\ell}{n}}\right)= n\rounddown{\dfrac{\rho_\ell}{n}} - \rho_\ell + r\left(\roundup{\dfrac{\rho_\ell}{n}} - \rounddown{\dfrac{\rho_\ell}{n}}\right),
\end{align*}
that
\begin{align*}
    \dg_{\tilde{\Theta}}(\ell) &\geq r\roundup{\dfrac{\rho_\ell}{n}} - |M_\ell| - \left(\rho_\ell - |M_\ell| - (n-r)\roundup{\dfrac{\rho_\ell}{n}}\right)= n\roundup{\dfrac{\rho_\ell}{n}} -\rho_\ell,
\end{align*}
and that
\begin{align*}
    \dg_{\tilde{\Theta}}(\ell) &\geq r\roundup{\dfrac{\rho_\ell}{n}} - |M_\ell| - \rounddown{\dfrac{\rho_\ell - |M_\ell| - d_\ell}{2}} = r\roundup{\dfrac{\rho_\ell}{n}} - \rounddown{\dfrac{\rho_\ell + |M_\ell| - d_\ell}{2}},
\end{align*}
satisfying the one upper bound and two lower bounds on $\dg_{\tilde{\Theta}}(\ell)$ resulting from \eqref{AHtildeThetaconditions}. Hence, we now have that $\tilde{\Theta}$ satisfies \eqref{AHtildeThetaconditions}. 

\noindent Conversely, suppose that $\tilde{\Theta}\subseteq\tilde{\Gamma}$ satisfying \eqref{AHtildeThetaconditions} exists.
As $\tilde{\Gamma}\subseteq\Gamma$, we have that $\tilde{\Theta}\subseteq\Gamma$. Let $\Theta$ be the subgraph of $\Gamma$ with edge multi-set $E(\Gamma)\backslash E(\tilde{\Theta})$. Then as
\begin{align*}
    \mult_{\tilde{\Gamma}}(x_i\ell) = 
    \begin{cases}
        1 &\text{ if }n\nmid \rho_\ell \text{ and }\dg_{F(\ell)}(x_i)<\roundup{\dfrac{\rho_\ell}{n}},\\
        0 &\text{ otherwise},
    \end{cases}
\end{align*}
it follows that $\mult_{\tilde{\Theta}}(x_i\ell)=0$ if $n\mid \rho_\ell$ and $\mult_{\tilde{\Theta}}(x_i\ell)\leq 1$ if $n\nmid \rho_\ell$. Hence,
\begin{align*}
    \mult_{\tilde{\Theta}}(x_i\ell) \leq \roundup{\rho_\ell/n}-\rounddown{\rho_\ell/n} \quad \text{ for }i\in[r],\ell\in[k],
\end{align*}
and so as $\dg_\Theta(x_i)=\dg_\Gamma(x_i) - \dg_{\tilde{\Theta}}(x_i)$ and $\dg_\Theta(\ell)=\dg_\Gamma(\ell) - \dg_{\tilde{\Theta}}(\ell)$ for $i\in[r], \ell\in[k]$, we have by \eqref{AHtildeThetaconditions} that $\Theta$ satisfies \eqref{AHThetaconditions}. This completes the proof of Claim 3.\\

Let $Y$ be the set of all free colors $\ell\in[k]$. Observe that the degree of each forced color $\ell\in[k]$ is $0$ in $\tilde{\Gamma}$. Hence, the existence of $\tilde{\Theta}\subseteq\tilde{\Gamma}$ is equivalent to the existence of $\tilde{\Theta}[X,Y]\subseteq\tilde{\Gamma}[X,Y]$, where $\tilde{\Theta}[X,Y]$ and $\tilde{\Gamma}[X,Y]$ are the subgraphs of $\tilde{\Theta}$ and $\tilde{\Gamma}$, respectively, induced by the vertex set $[X,Y]$.

Now we define the following functions.
\begin{align*}
    \begin{cases}
    g,f: X\cup Y\rightarrow \mathbb{N}\cup \{0\},\vspace{1mm}\\
    g(x_i)=f(x_i)=\displaystyle\sum_{\ell\in[k]}\roundup{\dfrac{\rho_\ell}{n}} - n& \text{ for } i\in [r],\\
    f(\ell)= n\roundup{\dfrac{\rho_\ell}{n}} - \rho_\ell &\text{ for } \ell\in Y,\vspace{1mm}\\
    g(\ell)= \max\Big\{n\rounddown{\dfrac{\rho_\ell}{n}}  + r\left(\roundup{\dfrac{\rho_\ell}{n}} - \rounddown{\dfrac{\rho_\ell}{n}}\right) - \rho_\ell, r\roundup{\dfrac{\rho_\ell}{n}}  - \rounddown{\dfrac{\rho_\ell +|M_\ell| - d_\ell}{2}}, 0\Big\} &\text{ for } \ell\in Y.
    \end{cases}
\end{align*}
Note that as $\sum_{\ell\in[k]}\rho_\ell = n^2$, we have that $f(x_i)= g(x_i)\geq 0$ for $i\in[r]$. Clearly, $f(\ell)\geq 0$ for $\ell\in Y$. We also have that for $\ell\in[k]$,
\begin{align*}
    (n-r)\Big(\roundup{\dfrac{\rho_\ell}{n}}-\rounddown{\dfrac{\rho_\ell}{n}}\Big)= n\roundup{\dfrac{\rho_\ell}{n}} - \rho_\ell -\Big(n\rounddown{\dfrac{\rho_\ell}{n}} + r\Big(\roundup{\dfrac{\rho_\ell}{n}}-\rounddown{\dfrac{\rho_\ell}{n}}\Big)\Big) \geq 0,
\end{align*}
and by \eqref{ncond1} we have that for $\ell\in[k]$,
\begin{align*}
    n\roundup{\dfrac{\rho_\ell}{n}} - \rho_\ell - r\roundup{\dfrac{\rho_\ell}{n}} + \rounddown{\dfrac{\rho_\ell + |M_\ell| - d_\ell}{2}} &\geq n\roundup{\dfrac{\rho_\ell}{n}} - \rho_\ell + \dfrac{1}{2}\left(\rho_\ell +|M_\ell|\right) - r\roundup{\dfrac{\rho_\ell}{n}} - \dfrac{1}{2}d_\ell\\
    &= (n-r)\roundup{\dfrac{\rho_\ell}{n}} - \dfrac{1}{2}d_\ell - \dfrac{1}{2}\left(\rho_\ell - |M_\ell|\right) \\&\geq 0
\end{align*}
and so $f(\ell)\geq g(\ell)$ for $\ell\in Y$. 
The existence of $\tilde{\Theta}[X,Y]\subseteq \tilde{\Gamma}[X,Y]$ for which \eqref{AHtildeThetaconditions} holds is equivalent to the existence of a $(g,f)$-factor in $\tilde{\Gamma}[X,Y]$. For $U\subseteq X$ and $\ell\in Y\backslash N_{\tilde{\Gamma}}(U)$, $\mult_{\tilde{\Gamma}}(\ell U) = 0$, and for $K\subseteq Y$ and $u\in X\backslash N_{\tilde{\Gamma}}(K)$, $\mult_{\tilde{\Gamma}}(u K) = 0$. Hence, we have by Theorem \ref{gffactorthm2} that $\tilde{\Gamma}[X,Y]$ has a $(g,f)$-factor if and only if the following conditions hold.
\begin{align*}
    \sum_{u \in U}g(u) &\leq \sum_{\ell\in Y} \min \Big\{ f(\ell), \mult_{\tilde{\Gamma}}(\ell U)\Big\} \quad &&\text{ for } U\subseteq X,\\
    \sum_{\ell\in K}g(\ell) &\leq \sum_{u\in X} \min \Big\{ f(u), \mult_{\tilde{\Gamma}}(u K)\Big\} \quad &&\text{ for } K\subseteq Y.
\end{align*}
As $\mult_{\tilde{\Gamma}}(\ell U) = \eta_I(\ell)$ for $I:=\{i\in[r] \ : \ x_i\in U\}$ and $\mult_{\tilde{\Gamma}}(uK) = \eta_K(i)$ for $i$ such that $u=x_i$, these two conditions are equivalent to the following.
\begin{align*}
    &|I|\Big(\sum_{\ell\in[k]}\roundup{\dfrac{\rho_\ell}{n}}-n\Big) \leq \sum_{\ell\in Y}\min\Big\{n\roundup{\dfrac{\rho_\ell}{n}} - \rho_\ell, \eta_I(\ell)\Big\} \quad &&\text{ for } I\subseteq[r],\\
    &\sum_{\ell\in K}\max\Big\{r\roundup{\dfrac{\rho_\ell}{n}}  - \rounddown{\frac{\rho_\ell + |M_\ell| - d_\ell}{2}}, r\Big(\roundup{\dfrac{\rho_\ell}{n}} - \rounddown{\dfrac{\rho_\ell}{n}}\Big) + n\rounddown{\dfrac{\rho_\ell}{n}} - \rho_\ell,0\Big\}\\
    &\quad\quad\quad \leq \sum_{i\in [r]} \min\Big\{\sum_{\ell\in[k]}\roundup{\dfrac{\rho_\ell}{n}} - n, \eta_K(i)\Big\}
    &&\text{ for } K\subseteq Y.\nonumber
\end{align*}
These are precisely the remaining necessary conditions \eqref{ncond4} and \eqref{ncond5} in Theorem \ref{mainthm}.
\end{proof}

\begin{remark}\label{rmkequivcond}\textup{
By Theorem \ref{gffactorthm1}, the following is also equivalent to the final two necessary conditions of Theorem \ref{mainthm}.
\begin{align*}
&\text{For }I\subseteq[r],K\subseteq[k],\\
    &(r-|I|)\Big(\sum_{\ell\in[k]}\roundup{\dfrac{\rho_\ell}{n}} - n\Big) + \sum_{\ell\in \bar{K}}\left(n\roundup{\dfrac{\rho_\ell}{n}} - \rho_\ell\right) \geq 
    \sum_{i\in I} \max\left\{ \sum_{\ell\in[k]}\roundup{\dfrac{\rho_\ell}{n}} - n - \eta_{K}(i),0 \right\} +\\
    &\quad \sum_{\ell\in K} \max\left\{ n\rounddown{\dfrac{\rho_\ell}{n}}  + r\left(\roundup{\dfrac{\rho_\ell}{n}} - \rounddown{\dfrac{\rho_\ell}{n}}\right) - \rho_\ell - \eta_{I}(\ell), r\roundup{\dfrac{\rho_\ell}{n}}  - \rounddown{\dfrac{\rho_\ell + |M_\ell| - d_\ell}{2}}-\eta_{I}(\ell), 0 \right\}.
\end{align*}
}
\end{remark}

\section{Corollaries} \label{s:cor}

Our main result implies the following result for constructing equitable $\brho$-Latin squares.
\begin{corollary} \label{corr=0}
    Let $k,n\in\mathbb{N}$ and $k\leq n^2$. Let $\brho=(\rho_1,\dots,\rho_k)$ such that $\rho_\ell\in\mathbb{N}$ for $\ell\in[k]$ and $\sum_{\ell\in[k]}\rho_\ell = n^2$, and let $\vb*{d}=(d_1,\dots,d_k)$ such that $d_\ell\in\mathbb{N}\cup\{0\}$ for $\ell\in[k]$ and $\sum_{\ell\in[k]}d_\ell \leq n$. Then it is possible to construct an $n\times n$ symmetric equitable $\brho$-Latin square on $[k]$ which has the prescribed diagonal $\vb*{d}$ if and only if the following hold.
\begin{align*}
    &d_\ell \leq \rho_\ell \quad \text{ for }\ell\in[k],\\
    &\Big|\{\ell\in[k] \ : \ \rho_\ell \nequiv d_\ell  \Mod 2\}\Big| \leq n - \sum_{\ell\in[k]}d_\ell,\\
    &n - \sum_{\ell\in[k]}d_\ell \equiv \Big|\{\ell\in[k] \ : \ \rho_\ell \nequiv d_\ell \Mod 2\}\Big| \Mod 2.
\end{align*}
\end{corollary}
\begin{proof}
    When $r=0$, conditions \eqref{ncond0}, \eqref{ncond2}, and \eqref{ncond3} of Theorem \ref{mainthm} simplify to the three necessary conditions of Corollary \ref{corr=0}, condition \eqref{ncond1} of Theorem \ref{mainthm} is trivially true, condition \eqref{ncond4} of Theorem \ref{mainthm} is vacuously true, and \eqref{ncond5} of Theorem \ref{mainthm} reduces to a trivial statement. Note that an $n\times n$ square which is symmetric off an empty $0\times 0$ rectangle is symmetric. By Theorem \ref{mainthm}, the result follows.
\end{proof}

As noted in the introduction, Andersen's theorem follows from our main result.
\begin{corollary}\cite[Theorem 11]{MR772580} \label{cor:andersen}
    Let $\vb*{d}=(d_1,\dots,d_n)$ such that $\sum_{\ell\in[n]}d_\ell\leq n-r$. Then an $r\times r$ externally symmetric Latin rectangle $M$ on $[n]$ can be extended to an $n\times n$ Latin square which is symmetric off $M$ and has the prescribed diagonal tail $\vb*{d}$ if and only if the following hold.
        \begin{align*}
            &|M_\ell| - d_\ell \geq 2r-n \quad &&\text{ for } \ell \in [n],\\
            &|M_\ell|+d_\ell \equiv n \Mod 2 \quad &&\text{ for at least } r + \sum_{\ell\in[n]}d_\ell \text{ symbols } \ell\in[n].
        \end{align*}
\end{corollary}
\begin{proof}
    First, recall that an externally symmetric Latin rectangle is an externally symmetric equitable $\brho$-Latin rectangle with $\brho=(n,\dots,n)$. Hence, for $\ell\in[n]$, we have that $\rho_\ell = n$, $|M_\ell|\leq r$, and $d_\ell\leq n-r$, so \eqref{ncond0} follows immediately and \eqref{ncond1} is equivalent to the following.
    \begin{align*}
        2r-n \leq |M_\ell| - d_\ell.
    \end{align*}
    We also have that
    \begin{align*}
        \Big|\{\ell\in[n] \ : \ \rho_\ell - |M_\ell| - d_\ell \not\equiv 0\Mod 2\}\Big| = n - \Big|\{\ell\in[n] \ : \ |M_\ell| + d_\ell \equiv n\Mod 2\}\Big|,
    \end{align*}
    so \eqref{ncond2} is equivalent to the following.
    \begin{align*}
        |M_\ell| + d_\ell \equiv n\Mod 2 \text{ for at least }r+\sum_{\ell\in[n]}d_\ell \text{ symbols }\ell\in[n],
    \end{align*}
   and it also follows that
    \begin{align*}
        n \equiv \sum_{\ell\in [k]}(|M_\ell|+d_\ell) + \Big|\{\ell \ : \ n \not\equiv |M_\ell|+d_\ell \Mod 2\}\Big| \Mod 2.
    \end{align*}
    As $\sum_{\ell\in[k]}|M_\ell| = r^2$ and $r^2\equiv r \Mod 2$, this is equivalent to \eqref{ncond3}.
    Condition \eqref{ncond4} of Theorem \ref{mainthm} reduces to a trivial statement when $\rho_\ell=n$ for $\ell\in[k]$. Moreover, since $\rho_\ell=n$ for $\ell\in[k]$, all colors in $[k]$ are forced colors, and so condition \eqref{ncond5} holds vacuously. Hence, we may apply Theorem \ref{mainthm} with $\brho=(n,\dots,n)$ to obtain the result.
\end{proof}

In our main result, we do not require all diagonal entries to be prescribed. The need to make a distinction between symbols for which $\rho_\ell$ is odd and symbols for which $\rho_\ell$ is even complicates the proof of our main result. As in Hoffman's theorem, we may instead completely prescribe the diagonal to avoid this complication. For $\vb*{d}$ such that $\sum_{\ell\in[k]}d_\ell=n-r$, conditions \eqref{ncond2} and \eqref{ncond3} simplify to $\rho_\ell - |M_\ell| - d_\ell \equiv 0\Mod 2$ for $\ell\in[k]$. On the other hand, as in Cruse's theorem, we may also leave all cells of the diagonal tail unprescribed (we have done so in the following corollary). For $\vb*{d}$ such that $\sum_{\ell\in[k]}d_\ell=0$, not only does \eqref{ncond0} become trivial, but the remaining necessary conditions of our main result are also simplified.

In the remainder of this section, we define $D$ to be the sum of $d_\ell$ over all $\ell\in[k]$ and $q$ to be the number of symbols such that $\rho_\ell + |M_\ell| + d_\ell$ is odd. In other words, 
\begin{align*}
    D=\sum_{\ell\in[k]}d_\ell,  \text{ and }  q = \left|\{\ell\in[k] \ : \ (\rho_\ell - |M_\ell| - d_\ell) \not\equiv 0\Mod 2\}\right|.
\end{align*} 
Note that \eqref{ncond2} and \eqref{ncond3} of Theorem \ref{mainthm} both hold if and only if $(n-r-D-q)/2$ is a non-negative integer. 

\begin{corollary} \label{cor2}
    An $r\times r$ externally symmetric equitable $\brho$-Latin rectangle $M$ on $[k]$ can be extended to an $n\times n$ equitable $\brho$-Latin square which is symmetric off $M$ if and only if the following conditions hold.
    \begin{align*}
        &(n-r)\rounddown{\dfrac{\rho_\ell}{n}} \leq \rho_\ell - |M_\ell| \leq 2(n-r)\roundup{\dfrac{\rho_\ell}{n}} &&\text{ for } \ell\in[k],\\
        &|I|\Big(\sum_{\ell\in[k]}\roundup{\dfrac{\rho_\ell}{n}}-n\Big) \leq \sum_{\ell\in[k]}\min\Big\{n\roundup{\dfrac{\rho_\ell}{n}}  - \rho_\ell, \eta_I(\ell)\Big\}  \quad &&\text{ for }I\subseteq[r],\\
        &\sum_{\ell\in K}\max\Big\{r\roundup{\dfrac{\rho_\ell}{n}} -  \rounddown{\dfrac{\rho_\ell + |M_\ell|}{2}}, r\Big(\roundup{\dfrac{\rho_\ell}{n}} - \rounddown{\dfrac{\rho_\ell}{n}}\Big) + n\rounddown{\dfrac{\rho_\ell}{n}} - \rho_\ell,0\Big\}\\
        &\quad \quad \quad \leq \sum_{i\in [r]} \min\Big\{\sum_{\ell\in[k]}\roundup{\dfrac{\rho_\ell}{n}} - n, \eta_K(i)\Big\} && \text{ for }K\subseteq Y,
\end{align*}
where $Y=\{\ell\in[k] \ : \ n\nmid \rho_\ell\}$, and $(n-r-q)/2$ is a non-negative integer.
\end{corollary}
\begin{proof}
    As $d_\ell=0$ for $\ell\in[k]$, \eqref{ncond0} is trivial by definition of equitable $\brho$-Latin rectangles. As $D=0$, \eqref{ncond2} and \eqref{ncond3} are simplified to the condition that $(n-r-q)/2$ be a non-negative integer. Finally, the remaining necessary conditions in Theorem \ref{mainthm} are simplified to the remaining necessary conditions in Corollary \ref{cor2}. We apply Theorem \ref{mainthm} with $\sum_{\ell\in[k]}d_\ell = 0$ to obtain this result.
\end{proof}

The final two necessary conditions of Theorem \ref{mainthm} are conditions that apply to all possible subsets of the sets of rows and free symbols, respectively, in a given equitable $\brho$-Latin rectangle. In the remainder of this section, our goal is to simplify these conditions in special cases. We note first of all that when a symbol is free, there is a choice to be made as to how many times it occurs in each row and column of an equalized $\brho$-latin square. By requiring all symbols to be forced, we eliminate this choice. As a result, the necessary conditions for the existence of the desired $(g,f)$-factor in the proof of Theorem \ref{mainthm} simplify a great deal, as in the following corollary.
\begin{corollary} \label{cor3}
Let $n\mid\rho_\ell$ for $\ell\in[k]$. Then an $r\times r$ externally symmetric equitable $\brho$-Latin rectangle $M$ on $[k]$ can be extended to an $n \times n$ equitable $\brho$-Latin square which is symmetric off $M$ and has the prescribed diagonal tail $\vb*{d}$ if and only if $(n-r-D-q)/2$ is a non-negative integer and the following hold.
\begin{align*}
    &\rho_\ell \geq |M_\ell| + d_\ell &&\text{ for }\ell\in[k],\\
    &(n-r)\dfrac{\rho_\ell}{n} \leq \rho_\ell - |M_\ell| \leq 2(n-r)\dfrac{\rho_\ell}{n} - d_\ell &&\text{ for }\ell\in[k].
\end{align*}
\end{corollary}
\begin{proof}
    Let $M$ be an $r\times r$ symmetric equitable $\brho$-Latin rectangle. Then if $n\mid \rho_\ell$, the first three necessary conditions of Corollary \ref{cor3} are equivalent to the necessary conditions for $(\brho,\vb*{d})$-admissibility and condition \eqref{ncond4} reduces to a trivial statement. Moreover, condition \eqref{ncond5} is vacuously true when all colors are forced. 
    We apply Theorem \ref{mainthm} with $\brho$ such that $n\mid \rho_\ell$ for $\ell\in[k]$ to obtain this result.
\end{proof}

\section{Concluding Remarks and Open Problems}\label{s:openp}

Externally symmetric and symmetric squares have the property that each symbol occurs exactly the same number of times in row $i$ as in column $i$. In order to make the squares under consideration ``equitable'' with respect to this property, we could allow the number of occurrences of a symbol in such pairs of corresponding rows and columns to differ by at most $1$, as in the following problem.
\begin{question}
    Let $M$ be an $r\times r$ equitable $\brho$-Latin rectangle such that
    \begin{align} \label{equitablysymdef}
        \Big||M_\ell^i|- |^iM_\ell|\Big|\leq 1 \quad \text{ for }\ell\in[k].
    \end{align}
    Find conditions that ensure $M$ can be extended to an
    \begin{enumerate}
        \item [\textup{(a)}] equitable $\brho$-Latin rectangle that is symmetric off $M$ and has the prescribed diagonal tail $\vb*{d}$.
        \item [\textup{(b)}] externally symmetric equitable $\brho$-Latin square that has the prescribed diagonal tail $\vb*{d}$.
    \end{enumerate}
\end{question}

A Sudoku square is a $9\times 9$ Latin square in which each $3\times 3$ sub-square contains exactly one copy of each symbol in $[9]$. It is not possible to construct a Sudoku square that is symmetric since the symmetry in the top left sub-square would force the structure to contain at least two copies of at least one symbol in every row and column as well as in every sub-square. However, if we allow multiple copies of symbols in rows and columns as in an equitable $\brho$-Latin square, then we may construct a Sudoku-like structure of this sort. For square numbers $n=m^2$, we say an $n\times n$ array $N$ is a symmetric equitable $\brho$-Sudoku square if each row, each column, and each $m\times m$ sub-square each contain approximately $\rho_\ell/n$ copies of each symbol $\ell\in[k]$. An example of one such square is in Table \ref{tab:Sudokuex}.
\begin{table}[h]
    \centering
        \begin{tabular}{|ccc|ccc|ccc|}
        \hline
        1 & 4 & 2 & 3 & 2 & 4 & 3 & 4 & 4\\
        4 & 3 & 4 & 2 & 1 & 4 & 4 & 3 & 2\\
        2 & 4 & 3 & 4 & 4 & 3 & 4 & 2 & 1\\
        \hline
        3 & 2 & 4 & 3 & 4 & 4 & 1 & 4 & 2\\
        2 & 1 & 4 & 4 & 3 & 2 & 4 & 3 & 4\\
        4 & 4 & 3 & 4 & 2 & 1 & 2 & 4 & 3\\
        \hline
        3 & 4 & 4 & 1 & 4 & 2 & 3 & 2 & 4\\
        4 & 3 & 2 & 4 & 3 & 4 & 2 & 1 & 4\\
        4 & 2 & 1 & 2 & 4 & 3 & 4 & 4 & 3\\
        \hline
        \end{tabular}
    \caption{A symmetric equitable $\brho:=(9,18,18,36)$-Sudoku square}
    \label{tab:Sudokuex}
\end{table}
\begin{question}
    Let $n$ be a square number. Under what conditions can an $r\times r$ symmetric equitable $\brho$-Latin $r\times r$ rectangle be extended to an $n\times n$ symmetric equitable $\brho$-Sudoku square?
\end{question}

\section{Acknowledgments}
We thank the anonymous referees for carefully reading the paper and for many detailed and helpful comments that improved the readability and accuracy of the manuscript, as well as for their patient and generous feedback throughout the review process.

\bibliographystyle{plain}
\bibliography{Bibliography}
\end{document}